\documentclass[11pt]{article}

\usepackage[a4paper,margin=2.5cm]{geometry}

\usepackage{amsmath,amssymb,amsfonts,amsthm,mathtools}

\usepackage{graphicx}
\usepackage{float}
\usepackage{booktabs}
\usepackage{multirow}
\usepackage{array}
\usepackage{graphicx}%
\usepackage{multirow}%
\usepackage{amsmath,amssymb,amsfonts}%
\usepackage{amsthm}%
\usepackage{mathrsfs}%
\usepackage[title]{appendix}%
\usepackage{xcolor}%
\usepackage{textcomp}%
\usepackage{manyfoot}%
\usepackage{booktabs}%
\usepackage{algorithm}%
\usepackage{algorithmicx}%
\usepackage{algpseudocode}%
\usepackage{listings}%
\usepackage{soul}
\usepackage{tikz}
\usetikzlibrary{arrows.meta}

\newcommand{\zdg}{\overrightarrow{\Gamma}}
\newcommand{\ezdg}{\overrightarrow{\Gamma}_{\!E}}

\newcommand{\adg}{\overrightarrow{\mathrm{AG}}}
\newcommand{\Annl}{\mathrm{Ann}_{l}}
\newcommand{\Annr}{\mathrm{Ann}_{r}}
\newcommand{\Ann}{\mathrm{Ann}}
\newcommand{\Nil}{\mathrm{Nil}}
\newcommand{\Diam}{\mathrm{Diam}}
\newcommand{\Girth}{\mathrm{Girth}}
\newcommand{\ZP}{\mathbb{Z}^{+}}
\newcommand{\kd}{\mathrm{kd}}
\usepackage{enumitem}
\usepackage{setspace}
\usepackage[hidelinks]{hyperref}

\newtheorem{theorem}{Theorem}[section]
\newtheorem{lemma}[theorem]{Lemma}
\newtheorem{proposition}[theorem]{Proposition}
\newtheorem{corollary}[theorem]{Corollary}

\theoremstyle{definition}
\newtheorem{definition}[theorem]{Definition}
\newtheorem{example}[theorem]{Example}
\newtheorem{remark}[theorem]{Remark}

\title{\Large\bfseries
Annihilator Digraphs and Extended Zero-Divisor Digraphs of Semigroups and Rings}

\author{
	Rasie Mekera$^{1}$,
	Defne Somer$^{2}$,
	Didem Ye\c{s}il$^{2}$	
}

\date{}

\begin{document}
	
	\maketitle
	
	\begin{center}
		\small
		$^{1}$Independent Researcher, Xanthi, Greece
		
		\vspace{0.4cm}
		
		$^{2}$Department of Mathematics,
		\c{C}anakkale Onsekiz Mart University,
		\c{C}anakkale, T\"urkiye
		
		\vspace{0.5cm}
		
		\texttt{rasiemekera@gmail.com}$^{*}$
		
		\texttt{defne.somer@comu.edu.tr}
		
		\texttt{dyesil@comu.edu.tr}
		
		\vspace{0.3cm}
		
		$^{*}$Corresponding author.
		
		
	\end{center}
	
	%------------------------------------------------
	
	\begin{abstract}
	Let $S$ be a semigroup with zero. This paper studies the zero-divisor digraph $\zdg(S)$ and the extended zero-divisor digraph $\ezdg(S)$, and introduces the annihilator digraph $\adg(S)$ via left and right annihilators. The diameter bound when every zero-divisor is nilpotent, and the sinks and the sources of $\ezdg(S)$ being identical to those of $\zdg(S)$ is demonstrated. The conditions in which $\zdg(S)=\ezdg(S)=\adg(S)$ holds are established, and the connectedness, diameter, girth, and vertex degrees of $\adg(S)$ are bounded when every zero-divisor is nilpotent or two-sided. The extended zero-divisor digraph is connected if and only if the zero-divisor digraph is connected, and it contains a directed cycle if and only if the zero-divisor digraph does. The knit degrees of $\zdg(S)$, $\ezdg(S)$, and $\adg(S)$ are computed. For a unital ring $R$, the equality $\ezdg(R)=\zdg(R)$ is characterized by nilpotency indices and one-sided annihilator conditions; it holds for the full matrix ring $M_{n}(F)$ over a field $F$ if and only if $n=2$, and $\adg(M_{n}(F))$ is connected and contains a directed cycle. Moreover, for an artinian noncommutative ring $R$, it was proved that $\zdg(R)$ is connected if and only if $\ezdg(R)$ is connected if and only if every one-sided identity element of $R$ is a two-sided identity of $R$.
	\end{abstract}
	
	\bigskip
	
	\noindent\textbf{Keywords:}
Zero-divisor digraph | extended zero-divisor digraph | annihilator digraph | semigroup | noncommutative ring
	
	\medskip
	
	
\section{Introduction}\label{sec:intro}

In \cite{b88}, Beck introduced the zero-divisor graph of a commutative ring $R$ as a graph whose vertex set consisted of the zero divisors of $R$, including zero. In \cite{al99}, the authors excluded zero from the vertex set and introduced the zero-divisor graph of commutative rings in its most widely used form, in which two distinct vertices $x$ and $y$ are adjacent if and only if $xy=0$. This graph is simple and undirected.

For noncommutative rings, Redmond introduced the zero-divisor digraph $\zdg(R)$ in \cite{r02}, in which there is an arc $x\to y$ if and only if $xy=0$. Wu \cite{w05} investigated the connectivity of $\zdg(R)$ for Artinian rings and characterized its sinks and sources, while Akbari and Mohammadian \cite{am06} obtained further results on zero-divisor graphs of noncommutative rings.

In parallel, DeMeyer et al. carried the zero-divisor graph to semigroups and introduced the zero-divisor graph of a commutative semigroup with zero \cite{dms02}, and subsequently \cite{dd05} contributed additional research.
Wright \cite{w07} defined the zero-divisor digraph $\zdg(S)$ of an arbitrary semigroup $S$ with zero, proved that the directed distance from a left zero-divisor to a right zero-divisor is at most $3$, and characterized the connectedness of $\zdg(S)$ by the property that every zero-divisor of $S$ is two-sided. 

Another approach modifies the adjacency relation of the zero-divisor graph. Bennis et al. \cite{bmt16} introduced the extended zero-divisor graph $\overline{\Gamma}(R)$ of a commutative ring with identity, in which distinct vertices $x$ and $y$ are adjacent if and only if $x^{n}y^{m}=0$ for some $n,m\in\ZP$ with $x^{n}\neq 0$ and $y^{m}\neq 0$, and Alkhamaiseh \cite{a20} studied the extended zero-divisor graphs of the rings of Gaussian integers modulo $n$. 
Badawi \cite{b14} defined the annihilator graph $AG(R)$, whose adjacency is determined by the annihilators of $x$, $y$, and $xy$. Afkhami et al. \cite{aks15} extended the annihilator graph to commutative semigroups. In \cite{trans}, Pookpienlert et al. also considered a graph called the annihilator graph for partial transformation semigroups, in which two vertices are adjacent if and only if they possess a common nonzero two-sided annihilator.
Further variants replace the element zero by an ideal or a congruence. Redmond \cite{r03} introduced the ideal-based zero-divisor graph, Anderson and Lewis \cite{al16} developed a general theory of congruence-based zero-divisor graphs, and Spiroff and Wickham \cite{sw11} introduced the compressed zero-divisor graph. Anderson and McClurkin \cite{am20} unified these constructions for commutative rings by studying the extended zero-divisor graph, the annihilator graph, and their ideal-based, congruence-based, and compressed analogs in a single framework.

In \cite{syc}, the authors analyzed the extended zero-divisor digraph $\ezdg(S)$ of a noncommutative semigroup $S$ with zero, characterized the semigroups satisfying $\ezdg(S)=\zdg(S)$, and then explored the completeness, diameter, and girth of $\ezdg(S)$ by means of nilpotency indices and annihilator sets. Motivated by \cite{syc}, this paper investigates the results of \cite{am20} for digraphs of semigroups and rings in light of \cite{syc} and obtains new results on extended zero-divisor digraphs of both rings and semigroups. This paper further bounds the distance between nilpotent vertices of $\zdg(S)$, identifies the end vertices of $\ezdg(S)$, and demonstrates that its sinks and sources are identical with those of $\zdg(S)$. The annihilator digraph $\adg(S)$ was introduced and compared with $\zdg(S)$ and $\ezdg(S)$. Its connectivity, diameter, girth, and vertex degrees were bounded when every zero divisor was either nilpotent or a two-sided zero divisor, and these parameters were computed for full matrix rings. The knit degree introduced in \cite{akk11} was extended from commuting graphs to the zero-divisor digraphs. Finally, the conditions under which the three digraphs coincide were described.
For unital rings, the equality $\ezdg(R)=\zdg(R)$ is characterized in a form strictly simpler than its semigroup counterpart in \cite{syc}, and it is shown to hold for the full matrix ring $M_{n}(F)$ over a field $F$ exactly when $n=2$. Necessary and sufficient conditions for the connectedness of $\ezdg(R)$ are obtained for Artinian rings. 

%%%%%%%%%%%%%%%%%%%%%%%%%%%%%%%%%%%%%%%%%%%%%%%%%%%%%%%%%%%%%%%%%%%%%%%%%%%%%%%%
\section{Preliminaries}\label{sec:prelim}

Throughout this paper, $S$ denotes a semigroup with zero, and $R$ denotes a ring. For any subset $A$ of either $S$ or $R$, $A^* $ denotes $A \setminus \{0\}$.

An element $x\in S$ is a \textbf{left zero divisor} if $xy=0$, and a \textbf{right zero divisor} if $yx=0$ for some $y\in S^*$. The set of all left or right zero divisors of $S$ is denoted by $Z(S)$ \cite{w07}. The set of nilpotent elements of $S$ is denoted by $\Nil(S)$, and for $x\in\Nil(S)^*$, the \textbf{nilpotency index} $n_{x}$ is the least positive integer with $x^{n_{x}}=0$, that is $x^{n_{x}-1}\neq 0$. The set of idempotent elements of $S$ is denoted by $E(S)$. For $x\in S$, the sets $\Annl(x)=\{s\in S: sx=0\}$ and $\Annr(x)=\{s\in S: xs=0\}$ are called the \textbf{left} and \textbf{right annihilators of $x$}, respectively \cite{semiring}. Moreover, $\Ann(x)=\Annl(x)\cap\Annr(x)$ \cite{trans}. For $A\subseteq S$, the \textbf{radical} of $A$ is $\sqrt{A}=\{s\in S: s^{n}\in A, \text{ for some } n\in\ZP\}$ \cite{radical}. A nonempty subset $I\subseteq S$ is an \textbf{ideal} if $SI\subseteq I$ and $IS\subseteq I$. If $(S, \cdot)$ is a semigroup, the semigroup $(S, \circ)$
where $x \circ y = yx$ is called the \textbf{opposite semigroup} and denoted by $S^{\mathrm{op}}$ \cite{h95, ganyus}. 

The definitions given below and all other relevant fundamental definitions, such as directed walks, paths, and cycles, spanning subdigraphs, and complete digraphs, are the same as those in \cite{bg09,syc}. 
All digraphs in this paper are simple digraphs, i.e., digraphs with no loops and no multi-arcs. A digraph $G$ is said to be \textbf{connected} if for every distinct vertices $x$ and $y$, there exists a directed path from $x$ to $y$. The \textbf{directed distance} $\overrightarrow{d}(x,y)$ is the length of the shortest directed path from $x$ to $y$. If no such path exists, then $\overrightarrow{d}(x,y)=\infty$. The \textbf{diameter} is the greatest directed distance between any vertices in $G$ and is denoted by $\Diam(G)$. The \textbf{girth} of $G$, denoted by $\Girth(G)$, is the length of the shortest directed cycle of $G$. If $G$ has no directed cycles, then $\Girth(G)=\infty$. The \textbf{out-degree} and the \textbf{in-degree} of a vertex are the numbers of arcs leaving and entering it, respectively. A vertex $v$ of a digraph is a \textbf{sink} if its out-degree is zero and its in-degree is positive, and a \textbf{source} if its in-degree is zero and its out-degree is positive. Altogether, sinks and sources are called \textbf{end vertices}. The \textbf{converse} of a digraph is the digraph with the same vertices and the reversed arcs.

\begin{definition}\label{def:zdg} \cite{w07} 
	The \textbf{zero-divisor digraph} of $S$ is the simple digraph $\zdg(S)$ with vertices $Z(S)^{*}$, and for distinct $x,y\in Z(S)^{*}$, $x\to y$ is an arc of $\zdg(S)$ if and only if $xy=0$.
\end{definition}

\begin{definition}\label{def:ezdg}
	\cite{syc} The \textbf{extended zero-divisor digraph} of $S$ is the simple digraph $\ezdg(S)$ with vertices $Z(S)^{*}$, and for distinct $x,y\in Z(S)^{*}$, $x\to y$ is an arc of $\ezdg(S)$ if and only if $x^{n}y^{m}=0$ for some $n,m\in\ZP$ with $x^{n}\neq 0$ and $y^{m}\neq 0$.
\end{definition}
If the semigroup with zero has no zero divisor under product, then in the zero-divisor digraph and in the extended zero-divisor digraph there are no vertices.
\begin{proposition}\cite[Proposition 3.2]{syc}\label{propsub}
	$\overrightarrow{\Gamma}(S)$ is a spanning subdigraph of 	$\overrightarrow{\Gamma}_E (S)$.
\end{proposition}

\begin{lemma}\label{lem:stab}
	\cite[Lemma 3.3]{syc} Let $x\in S^*$ with $x\notin\Nil(S)$. Then $\Annl(x^{2})=\Annl(x)$ if and only if $\Annl(x^{n})=\Annl(x)$, and $\Annr(x^{2})=\Annr(x)$ if and only if $\Annr(x^{n})=\Annr(x)$, for all $n\geq 2$.
\end{lemma}

\begin{theorem}\label{thm:syc-eq}
	\cite[Theorem 3.6]{syc}	
	The following are equivalent.
	\begin{enumerate}
		\item $\overrightarrow{\Gamma}_E (S) = \overrightarrow{\Gamma}(S)$ 
		\item If $Nil(S) \neq \{0\}$, then for all $x \in Nil(S)$, $n_x \leq 3$ and if $n_x = 3$ then $Ann_{l}(x^{2}) \setminus Ann_{l}(x) = \{x\}$ and $Ann_{r}(x^{2}) \setminus Ann_{r}(x) = \{x\}$. Moreover $Ann_l(y^2) = Ann_l(y)$ and $Ann_r(y^2) = Ann_r(y)$ for all non-zero $y \in Z(S) \setminus Nil(S)$. 
		\item If $Nil(S) \neq \{0\}$, then for all $x \in Nil(S)$, $n_x \leq 3$ and if $n_x = 3$ then $Ann_{l}(x^{2}) \setminus Ann_{l}(x) = \{x\}$ and $Ann_{r}(x^{2}) \setminus Ann_{r}(x) = \{x\}$. Moreover $\sqrt{Ann_{l}(y)} \setminus Nil(S) \subseteq Ann_{l}(y)$ and $\sqrt{Ann_{r}(y)} \setminus Nil(S) \subseteq Ann_{r}(y)$ for all  $y \in Z(S)^*$.
	\end{enumerate}
\end{theorem}
\begin{corollary}\label{ggg}\cite[Corollary 3.7]{syc}
	Let $S$ be a non-commutative semigroup with zero. Then, $\overrightarrow{\Gamma}_E(S) \neq \overrightarrow{\Gamma}(S)$ if there exist an $x \in Nil(S)$ such that $n_x \geq 4$.
\end{corollary}
\begin{proposition}\label{prop2}\cite[Proposition 2]{w07}
	The directed distance between any two vertices in $\zdg(S)$ is at most $3$, and $\zdg(S)$ is connected if and only if all zero divisors of $S$ are two-sided.
\end{proposition}
\begin{theorem}\label{theo4}\cite[Theorem 4]{w07}
	If an arc of $\zdg(S)$ belongs to a directed cycle of length greater than 4, then it must belong to a directed cycle of length 4 or less. If neither of the arc’s vertices satisfy $x^2=0$, or if the reverse arc doesn’t belong to $\zdg(S)$, then this cycle can be chosen to have length no less than 3.
\end{theorem}
\begin{lemma}\label{lem2}\cite[Example 2.8]{w05}
	For any field $F$, let $R$ be the ring of $n\times n$ $(n>1)$ matrices over $F$. Then the
	diameter of $\zdg(R)$ is 2.
\end{lemma}
\begin{theorem}\label{theo24}\cite[Theorem 2.4]{w05}
	Let $R$ be an artinian ring. Then, $\zdg(R)$ is connected if and only if every one-sided identity element of $R$ is the two-sided identity of $R$.
\end{theorem}
%%%%%%%%%%%%%%%%%%%%%%%%%%%%%%%%%%%%%%%%%%%%%%%%%%%%%%%%%%%%%%%%%%%%%%%%%%%%%%%%
\section{Distances and End Vertices}\label{sec:distance}

This section bounds the directed distance between nilpotent vertices of the zero-divisor digraph $\zdg(S)$. Moreover, this section presents a diameter bound by replacing that of Proposition \ref{prop2} assumption that every zero divisor is two-sided with the stronger assumption that every zero divisor is nilpotent, and proves that the sinks and sources of $\zdg(S)$ persist in $\ezdg(S)$.

\begin{theorem}\label{thm:nildist}
	Let $x,y\in\Nil(S)^{*}$ with $x\neq y$. Then
	$$
	\overrightarrow{d}(x,y)\leq 2
	$$
	in $\zdg(S)$.
\end{theorem}

\begin{proof}
	If $xy=0$, then $x\to y$ is an arc in $\zdg(S)$ and $\overrightarrow{d}(x,y)=1$. Suppose that $xy\neq 0$. Define $z_{0}=x^{n_{x}-1}$ and $z_{j}=x^{n_{x}-1}y^{j}$, for all $j\in\{1,2,\ldots,n_{y}\}$. Since $z_{0}=x^{n_{x}-1}\neq 0$ and $z_{n_{y}}=x^{n_{x}-1}y^{n_{y}}=0$, there is a greatest index $k\in\{0,1,\ldots,n_{y}-1\}$ with $z_{k}\neq 0$. Then
	$$
	xz_{k}=x^{n_{x}}y^{k}=0 \quad\text{and}\quad z_{k}y=z_{k+1}=0
	$$
	and $z_{k}y=0$ follows from the maximality of $k$. Since $xz_{k}=0$ with $x\neq 0$, the element $z_{k}$ belongs to $Z(S)^{*}$. If $z_{k}=x$, then $z_{k}y=0$ yields $xy=0$, and if $z_{k}=y$, then $xz_{k}=0$ yields $xy=0$. Both contradict $xy\neq 0$. Hence, $z_{k}\notin\{x,y\}$, and
	$$
	x\to z_{k}\to y
	$$
	is a directed path of length $2$ in $\zdg(S)$. Thus, $\overrightarrow{d}(x,y)\leq 2$. 
\end{proof}

\begin{corollary}\label{cor:nildiam}
	If $Z(S)=\Nil(S)\neq\{0\}$, then $\zdg(S)$ is connected. If $|Z(S)^{*}|\geq 2$, then
	$$
	\Diam(\zdg(S))\leq 2
	$$
\end{corollary}

\begin{proof}
	If $Z(S)=\Nil(S)$, then for any distinct $x,y\in Z(S)^{*}$, $\overrightarrow{d}(x,y)\leq 2$, by Theorem \ref{thm:nildist}. Thus, $\Diam(\zdg(S))\leq 2$.
\end{proof}

Since every ring is a semigroup with zero under its multiplicative operation, all results established for semigroups are also valid for rings. Therefore, every ring $R$ with $Z(R)=\Nil(R)\neq\{0\}$ is connected with $\Diam(\zdg(R))\leq 2$ whenever $|Z(R)^{*}|\geq 2$. 

\begin{theorem}\label{thm:sinks}
	The vertex $x$ is a sink (source) of $\ezdg(S)$ if and only if it is a sink (source) of $\zdg(S)$.
\end{theorem}

\begin{proof}
	($\Rightarrow:$) Let $r$ be a sink of $\ezdg(S)$. Since $\zdg(S)$ is a spanning subdigraph of $\ezdg(S)$, the out-degree of $r$ in $\zdg(S)$ is zero, and it remains to prove that the in-degree of $r$ in $\zdg(S)$ is positive. Suppose that $r$ is isolated in $\zdg(S)$, i.e., $rs=0$ implies $s\in\{0,r\}$ and $sr=0$ implies $s\in\{0,r\}$, for all $s\in S$. Since $r\in Z(S)^{*}$, there exist a $y \in Z(S)^{*}$ such that $r \in Ann_r(y)$ or $r \in Ann_l (y)$. However the only remaining possibility is $y =r$, thus $r^2=0$. Since the in-degree of $r$ in $\ezdg(S)$ is positive, there exist an arc $y\to r$, implying there exist $a,b \in \mathbb{Z}^+$ such that $y^{b}r^{a}=0$ with $r^{a}\neq 0$ and $y^{b}\neq 0$. Hence $a=1$. Consequently, $y^{b}r=0$, and $y^{b}=r$ since $r$ is isolated. As a result
	$$
	ry^{b}=r^{2}=0
	$$
	is obtained with $r\neq 0$ and $y^{b}=r\neq 0$. Then, $r\to y$ is an arc in $\ezdg(S)$, which contradicts the fact that out-degree of $r$ in $\ezdg(S)$ is zero. Hence, $r$ is not isolated in $\zdg(S)$, then there exist an arc whose head is $r$. Thus, $r$ is a sink in $\zdg(S)$. The converse statement for sources follows by a symmetric argument.	
	
	($\Leftarrow:$) Let $r$ be a sink of $\zdg(S)$. If $rs=0$ with $s\notin\{0,r\}$, then $s\in Z(S)^{*}$. Therefore, $r\to s$ is an arc in $\zdg(S)$. This is a contradiction since $r$ is a sink. Thus, $rs=0$ implies $s\in\{0,r\}$, for all $s\in S$. Suppose that $r\to y$ is an arc in $\ezdg(S)$ for some vertex $y\neq r$. Then, there exist  $a,b\in\ZP$ such that $r^{a}y^{b}=0$, $r^{a}\neq 0$, and $y^{b}\neq 0$ where $a$ and $b$ are least. If $a\geq 2$, then $r(r^{a-1}y^{b})=0$ yields $r^{a-1}y^{b}=r$ since $r$ is a sink and $r^{a-1}y^{b}=0$ contradicts with minimality of $a$. Multiplying $r^{a-1}y^{b}=r$ on the left by $r$, then $r^{2}=r^{a}y^{b}=0$ is obtained, which contradicts $r^{a}\neq 0$. Hence, $a=1$. Thus, $ry^{b}=0$, so $y^{b}=r$ since $y^{b}\neq 0$. If $b=1$, then $r \to y$ is an arc in $\zdg(S)$, which is a contradiction since $r$ is a sink of $\zdg(S)$. Hence, $b\geq 2$ and $ry\neq 0$. Therefore,
	$$
	y^{b+1}=y^{b}y=ry\neq 0 \text{ and } y^{2b}=r^{2}=0
	$$
	holds. Since $3b-1> 2b$
	$$
	ry^{2b-1}=y^{b}y^{2b-1}=y^{3b-1}=0
	$$
	so $y^{2b-1}\in\{0,r\}$. If $y^{2b-1}=0$, then $ry^{b-1}=y^{b}y^{b-1}=y^{2b-1}=0$ yields $y^{b-1}=r$ since $y^{b-1}\neq 0$, namely  $y^{b-1}=r=y^{b}$. Multiplying by $y$ repeatedly yields $y^{b}=y^{2b}=0$, which contradicts $y^{b}=r\neq 0$, then $y^{2b-1}=r$. Accordingly, $y^{2b-1}=y^{b}$, and multiplying on the left by $y$ yields $0=y^{2b}=y^{b+1}$, which contradicts $y^{b+1}\neq 0$. Since both cases result in a contradiction, it follows that the out-degree of $r$ in $\ezdg(S)$ is zero. The in-degree of $r$ in $\ezdg(S)$ is positive since $\zdg(S)$ is a spanning subdigraph of $\ezdg(S)$. Consequently, $r$ is a sink of $\ezdg(S)$. The converse statement for sources again follows by a symmetric argument.
\end{proof}
%%%%%%%%%%%%%%%%%%%%%%%%%%%%%%%%%%%%%%%%%%%%%%%%%%%%%%%%%%%%%%%%%%%%%%%%%%%%%%%%
\section{The Annihilator Digraph}\label{sec:adg}

This section introduces the annihilator digraph of a semigroup with zero, which is the directed analog of the annihilator graph defined in \cite{b14} for commutative rings and in \cite{aks15} for commutative semigroups, and compares it with $\zdg(S)$ and $\ezdg(S)$.

\begin{definition}\label{def:adg}
	The \textbf{annihilator digraph} of $S$ is the simple digraph $\adg(S)$ with vertex set $Z(S)^{*}$, where for distinct $x,y\in Z(S)^{*}$, $x\to y$ is an arc if and only if
	$$
	\Annl(xy)\setminus\big(\Annl(x)\cup\Annl(y)\big)\neq\emptyset
	$$ or
	$$
	\Annr(xy)\setminus\big(\Annr(x)\cup\Annr(y)\big)\neq\emptyset
	$$
\end{definition}
Just like with the zero-divisor digraph and the extended zero-divisor digraph, if there are no zero divisor elements in the semigroup, there are no vertices in the annihilator digraph. Moreover, since each of $\zdg(S)$, $\ezdg(S)$, and $\adg(S)$ have the same vertex set, all of the above theorems hold even when the graphs are edgeless.
\begin{remark}\label{rem:adg-comm}
	Let $S$ be a commutative semigroup with zero. Then for all $x,y \in S$, $\Annl(x)=\Annr(x)=\Ann(x)$ and $\Ann(x)\cup\Ann(y)\subseteq\Ann(xy)$. Thus, there is an arc $x\to y$ in $\adg(S)$ if and only if $\Ann(xy)\neq\Ann(x)\cup\Ann(y)$, if and only if there is an arc $y\to x$. Hence, for commutative semigroups, $\adg(S)$ is a symmetric digraph. If $S$ is noncommutative, then $\Annl(y)\subseteq\Annl(xy)$ may not hold. Hence, the adjacency is defined by the nonemptiness of the two set differences. 
\end{remark}

\begin{proposition}\label{prop:duality}
	The graphs $\zdg(S^{\mathrm{op}})$, $\ezdg(S^{\mathrm{op}})$, and $\adg(S^{\mathrm{op}})$ are converses of $\zdg(S)$, $\ezdg(S)$, and $\adg(S)$, respectively.
\end{proposition}

\begin{proof}
	The zero element and the powers of elements of $S^{\mathrm{op}}$ coincide with those of $S$, and $Z(S)^* = Z(S^{\mathrm{op}})^*$. Assume $x,y \in Z(S)^*$ are distinct. First, let $x \to y$ be an arc of $\zdg(S)$. Then, $xy=0=y \circ x$. Thus, $y \to x$ is an arc of $\zdg(S^{\mathrm{op}})$. If $x \to y$ is an arc of $\ezdg(S)$ then there exist positive integers $m,n \in \mathbb{Z}^+$ such that $x^my^n=0$ with $x^m \neq 0$ and $y^n \neq 0$. Hence, $y^n \circ x^m = 0$, this implies that $y \to x$ is an arc of $\ezdg(S^{\mathrm{op}})$. Lastly, let $x\to y$ be an arc of $\adg(S)$. Without the loss of generality, let $z \in \Annr(xy)\setminus\big(\Annr(x)\cup\Annr(y)\big)$. Then, $(xy)z=0$ with $xz \neq 0$ and $yz \neq 0$. Herefrom, $z \circ (y \circ x)=0$ with $z \circ x \neq 0$ and $z \circ y \neq 0$. This requires that $z \in \Annl(y \circ x)\setminus\big(\Annl(y)\cup\Annl(x)\big)$ and $y \to x$ is an arc of $\adg(S^{\mathrm{op}})$.
\end{proof}

\begin{definition}
	The graphs defined below have the vertex set $Z(S)^*$.
	\begin{enumerate}
		\item $x \to y$ is an arc of $\adg^{l}(S)$ if and only if $\Annl(xy)\setminus\big(\Annl(x)\cup\Annl(y)\big)\neq\emptyset$.
		\item $x \to y$ is an arc of $\adg^{r}(S)$ if and only if $\Annr(xy)\setminus\big(\Annr(x)\cup\Annr(y)\big)\neq\emptyset$.
		\item $\adg^{\wedge}(S)=\adg^{l}(S)\cap\adg^{r}(S)$.
	\end{enumerate}
\end{definition}

\begin{remark}\label{prop:variants}
	If $S$ is commutative, then $\adg^{l}(S)=\adg^{r}(S)=\adg^{\wedge}(S)=\adg(S)$. Otherwise,
	$\adg^{l}(S)=\adg^{r}(S)$ may not hold. Indeed, let $S=T_{3}(\mathbb{Z}_{2})$, where $T_{3}(\mathbb{Z}_{2})$ denotes the ring of $3\times 3$ upper triangular matrices over $\mathbb{Z}_{2}$, and let $$x=\begin{bmatrix}
		\overline{0} & \overline{0} & \overline{0} \\
		\overline{0} & \overline{1} & \overline{0} \\
		\overline{0} & \overline{0} & \overline{1}
	\end{bmatrix} \text{ and } y=\begin{bmatrix}
		\overline{0} & \overline{0} & \overline{1} \\
		\overline{0} & \overline{0} & \overline{0} \\
		\overline{0} & \overline{0} & \overline{1}
	\end{bmatrix}$$ Clearly, $x,y\in Z(S)^*$ since $xe_{11}=[\overline{0}]=ye_{11}$ where $e_{ij}$ represents the matrix whose $(i,j)$-entry is $\overline{1}$ and all other entries are $\overline{0}$. Then $xy=e_{33}$. Furthermore, \\ 
	$$z=\begin{bmatrix}
		\overline{1} & \overline{1} & \overline{0} \\
		\overline{0} & \overline{0} & \overline{0} \\
		\overline{0} & \overline{0} & \overline{0}
	\end{bmatrix} \in \Annl(xy)\setminus\big(\Annl(x)\cup\Annl(y)\big)$$ because
	$z(xy)=ze_{33}=[\overline{0}], \text{ } zx=e_{12}\neq [\overline{0}],\text{ } zy=e_{13}\neq [\overline{0}]$. Hence, $x\to y$ is an arc in $\adg^{l}(S)$. Moreover, let $w\in\Annr(xy)=\Annr(e_{33})$. Then third row of $w$ consists of only zeros. Thus, $yw=[\overline{0}]$, namely $w\in \Annr(y)$. Thus, $\Annr(xy)\setminus (\Annr(y)\cup \Annr(x))=\emptyset$. Therefore, $x\to y$ is not an arc of $\adg^{r}(S)$. 
\end{remark}

\begin{theorem}\label{thm:zdg-in-adg}
	Let $Z(S)\neq S$. Then $\zdg(S)$ is a spanning subdigraph of $\adg^{\wedge}(S)$. In particular, $\zdg(S)$ is a spanning subdigraph of $\adg(S)$.
\end{theorem}

\begin{proof}
	For distinct $x,y \in  Z(S)^*$, let $x\to y$ be an arc in $\zdg(S)$. Since $Z(S)\neq S$, there is an $s\in S\setminus Z(S)$. Then $s(xy)=s0=0$, so $s\in\Annl(xy)$. If $sx=0$, then $s$ is a left zero divisor since $x\neq 0$, which contradicts $s\notin Z(S)$. Accordingly, $sx\neq 0$, and similarly $sy\neq 0$. Hence,
	$s\in\Annl(xy)\setminus\big(\Annl(x)\cup\Annl(y)\big)
	\text{ and }s\in\Annr(xy)\setminus\big(\Annr(x)\cup\Annr(y)\big)$
	Consequently, $x\to y$ is an arc in $\adg^{\wedge}(S)$.
\end{proof}

\begin{corollary}\label{cor:monoid-adg}
	Let $S$ be a monoid with zero. Then $\zdg(S)$ is a spanning subdigraph of $\adg(S)$.
\end{corollary}

\begin{proof}
	Since $Z(S)\neq S$, Theorem \ref{thm:zdg-in-adg} applies.
\end{proof}

\begin{example}\label{ex:null}
	The hypothesis $Z(S)\neq S$ in Theorem \ref{thm:zdg-in-adg} cannot be omitted. For example, let $S=\{0,x,y\}$ be the null semigroup for which $\zdg(S)$ is the complete digraph on two vertices. Whereas $\Annl(s)=\Annr(s)=S$, for all $s\in S$. Namely, $\adg(S)$ has no arcs, and $\zdg(S)\not\subseteq\adg(S)$. Both digraphs are drawn in Figure \ref{fig:null}.
	\begin{figure}[htbp!]
		$$\begin{tikzpicture}[>=stealth, semithick]
			% Noktalar? tan?mlayal?m
			\coordinate (x) at (0,0);
			\coordinate (y) at (2.2,0);   % Sa? üst nokta  % Sol üst nokta
			
			% Do?ru parçalar?n? çiz ve uç noktalar?na etiket ekle
			\fill (x) circle (1.3pt) node[below] {$x$};
			\fill (y) circle (1.3pt) node[below] {$y$};
			
			% Noktalar? dolu daire ile vurgula
			\draw[shorten >=0.1cm,shorten <=0.01cm,->] (x) to[bend left=30] (y);
			\draw[shorten >=0.1cm,shorten <=0.01cm,->] (y) to[bend left=30] (x);
			\node[draw=none] at (1.1,-1.1) {$\zdg(S)$};
			\begin{scope}[xshift=5.2cm]
				\fill (0,0) circle (1.3pt) node[below] {$x$};
				\fill (2.2,0) circle (1.3pt) node[below] {$y$};
				\node[draw=none] at (1.1,-1.1) {$\adg(S)$};
			\end{scope}
		\end{tikzpicture} $$
		\caption{The digraphs $\zdg(S)$ and $\adg(S)$ of $S$.}
		\label{fig:null}
	\end{figure}
\end{example}

By Proposition \ref{propsub}, the digraph $\zdg(S)$ is always a spanning subdigraph of $\ezdg(S)$. Similarly, by Theorem \ref{thm:zdg-in-adg}, the digraph $\zdg(S)$ is also a spanning subdigraph of $\adg(S)$, provided that $Z(S)\neq S$. However, the digraph $\ezdg(S)$ is not necessarily a spanning subdigraph of $\adg(S)$. The following example demonstrates that this inclusion does not hold in general.

\begin{example}\label{ex:t3}
	Let $S=T_{3}(\mathbb{Z}_{2})$, and let
	$$
	x=\begin{bmatrix}
		\overline{1} & \overline{0} & \overline{0} \\
		\overline{0} & \overline{0} & \overline{1} \\
		\overline{0} & \overline{0} & \overline{0}
	\end{bmatrix} \quad\text{and}\quad y=\begin{bmatrix}
		\overline{0} & \overline{1} & \overline{0} \\
		\overline{0} & \overline{0} & \overline{0} \\
		\overline{0} & \overline{0} & \overline{1}
	\end{bmatrix}
	$$
	Then $xy=\begin{bmatrix}
		\overline{0} & \overline{1} & \overline{0} \\
		\overline{0} & \overline{0} & \overline{1} \\
		\overline{0} & \overline{0} & \overline{0}
	\end{bmatrix}\neq [\overline{0}]$, and
	$$
	x^{2}=e_{11},\qquad y^{2}=e_{33},\qquad x^{2}y^{2}=e_{11}e_{33}=[\overline{0}]
	$$
	with $x^{2}\neq [\overline{0}]$ and $y^{2}\neq [\overline{0}]$. Moreover, $x(xy^{2})=x^{2}y^{2}=[\overline{0}]$ with $xy^{2}=xe_{33}=e_{23}\neq [\overline{0}]$, and $(x^{2}y)y=[\overline{0}]$ with $x^{2}y=e_{11}y=e_{12}\neq [\overline{0}]$. Hence, $x,y\in Z(S)^{*}$, and $x\to y$ is an arc in $\ezdg(S)$.
	
	Let $z=\begin{bmatrix}
		\overline{a} & \overline{b} & \overline{c} \\
		\overline{0} & \overline{d} & \overline{e} \\
		\overline{0} & \overline{0} & \overline{f}
	\end{bmatrix}\in \Annl(xy)$. The product $z(xy)=\begin{bmatrix}
		\overline{0} & \overline{a} & \overline{b} \\
		\overline{0} & \overline{0} & \overline{d} \\
		\overline{0} & \overline{0} & \overline{0}
	\end{bmatrix}=[\overline{0}]$. Thus, $\overline{a}=\overline{b}=\overline{d}=\overline{0}$. From here,
	$$
	zx=\begin{bmatrix}
		\overline{a} & \overline{0} & \overline{b} \\
		\overline{0} & \overline{0} & \overline{d} \\
		\overline{0} & \overline{0} & \overline{0}
	\end{bmatrix}=[\overline{0}]
	$$
	This implies that $z\in\Annl(x)$, and hence $\Annl(xy)\setminus(\Annl(x)\cup\Annl(y))=\emptyset$. Similarly, if $w\in\Annr(xy)$, then $w\in\Annr(y)$. As a result, $\Annr(xy)\setminus(\Annr(x)\cup\Annr(y))=\emptyset$. Consequently, $x\to y$ is not an arc in $\adg(S)$, and
	$\ezdg(S)$ is not a subdigraph of $\adg(S)$.
	Note that $S$ is a finite monoid. Then, by Corollary \ref{cor:monoid-adg}, $\zdg(S)$ is a spanning subdigraph of $\adg(S)$.
\end{example}

\begin{proposition}\label{prop:onesided}
	Let $x,y\in Z(S)^{*}$ and $x\neq y$.
	\begin{enumerate}
		\item If $x^{n}y=0$ for some $n\geq 2$ with $x^{n}\neq 0$ and $x^{n-1}y\neq 0$, then $x\to y$ is an arc in $\adg(S)$.
		\item If $xy^{m}=0$ for some $m\geq 2$ with $y^{m}\neq 0$ and $xy^{m-1}\neq 0$, then $x\to y$ is an arc in $\adg(S)$.
	\end{enumerate}
\end{proposition}
\begin{proof}
	Let $x,y\in Z(S)^{*}$ and $x\neq y$.	
	\begin{enumerate}
		\item Let $n\geq 2$. Then, $x^{n-1}(xy)=x^{n}y=0$ holds. Therefore, by hypothesis $x^{n-1}\in\Annl(xy)\setminus(\Annl(x)\cup\Annl(y))$.
		\item Let $m\geq 2$ and $y^{m}\neq 0$. Since $0=xy^{m}=(xy)y^{m-1}$, $y^{m-1}\in\Annr(xy)\setminus(\Annr(x)\cup\Annr(y))$.
	\end{enumerate}
\end{proof}

\begin{theorem}\label{thm:2abs}
	If $abc=0$ implies $ab=0$, $bc=0$, or $ac=0$, for all $a,b,c\in S$, then
	\begin{enumerate}
		\item $\adg(S)$ is a subdigraph of $\zdg(S)$.
		\item $\ezdg(S)=\zdg(S)$.
	\end{enumerate}
\end{theorem}
\begin{proof}
	Let $x,y\in Z(S)^{*}$ and $x\neq y$.
	\begin{enumerate}
		\item Let $x\to y$ be an arc in $\adg(S)$. Then there exists $z\in S$ such that $zxy=0$, $zx\neq 0$, and $zy\neq 0$ or $xyz=0$, $xz\neq 0$, and $yz\neq 0$. If $zxy=0$, $zx\neq 0$, and $zy\neq 0$, by hypothesis $zx=0$, $xy=0$, or $zy=0$ must hold. Consequently, $xy=0$. Similarly $xy=0$ if $xyz=0$, $xz\neq 0$, and $yz\neq 0$. In both cases, $x\to y$ is an arc in $\zdg(S)$.
		\item Let $x\to y$ be an arc in $\ezdg(S)$. Then there exists $m,n\in \mathbb{Z}^{+}$ where $x^{n}y^{m}=0$, $x^{n}\neq 0$, and $y^{m}\neq 0$ such that $n$ and $m$ are least. If $n\geq 2$, from hypothesis $xx^{n-1}=x^{n}=0$, $x^{n-1}y^{m}=0$, or $xy^{m}=0$ is obtained. Hence, $xy^{m}=0$, so $n=1$ by minimality. If $m\geq 2$, from assumption, $xy=0$, $yy^{m-1}=y^{m}=0$, or $xy^{m-1}=0$. Thus, $xy=0$, implying $x\to y$ is an arc in $\zdg(S)$. From Proposition \ref{propsub}, $\zdg(S)$ is a spanning subdigraph of $\ezdg(S)$. Therefore, $\ezdg(S)=\zdg(S)$.
	\end{enumerate}
\end{proof}
\begin{corollary}\label{cor:2abs-equal}
	Let $Z(S)\neq S$, and let $abc=0$ implies $ab=0$, $bc=0$, or $ac=0$, for all $a,b,c\in S$. Then
	$$
	\zdg(S)=\ezdg(S)=\adg(S)
	$$
\end{corollary}
\begin{proof}
	Theorem \ref{thm:zdg-in-adg} and Theorem \ref{thm:2abs}(i) yield $\adg(S)=\zdg(S)$, and Theorem \ref{thm:2abs}(ii) completes the proof.
\end{proof}
\begin{example}\label{ex:converse-fails}
	The converse of Corollary \ref{cor:2abs-equal} does not hold. Let $S=\{0,1,a,b\}$ be the commutative monoid with identity $1$ and multiplication determined by
	$$
	a^{2}=b,\qquad ab=ba=0,\qquad b^{2}=0
	$$
	Then $Z(S)^{*}=\{a,b\}$, and $ab=0=ba$ yields that $\zdg(S)$ is the complete digraph with two vertices. Moreover, $\zdg(S)=\ezdg(S)=\adg(S)$. But, $a^{3}=a\,a^{2}=ab=0$, whereas $a^{2}=b\neq 0$.
\end{example}
%%%%%%%%%%%%%%%%%%%%%%%%%%%%%%%%%%%%%%%%%%%%%%%%%%%%%%%%%%%%%%%%%%%%%%%%%%%%%%%%
\section{Parameters of the Digraphs $\zdg(S)$, $\ezdg(S)$ and $\adg(S)$}\label{sec:params}
This section bounds the connectedness, diameter, girth, and vertex degrees of the digraphs $\zdg(S)$, $\ezdg(S)$, and $\adg(S)$, with emphasis on the annihilator digraph, and carries the knit degree of \cite{akk11} from commuting graphs to the zero-divisor digraphs. 
\begin{proposition}\label{prop:adg-diam}
	Let $Z(S)\neq S$ and $Z(S)=\Nil(S)\neq\{0\}$.
	If $|Z(S)^{*}|\geq 2$, then $\adg(S)$ is connected and $\Diam(\adg(S))\leq 2$.
\end{proposition}
\begin{proof}
	By Theorem \ref{thm:zdg-in-adg}, $\zdg(S)$ is a spanning subdigraph of $\adg(S)$, and by Corollary \ref{cor:nildiam}, $\zdg(S)$ is connected with diameter at most $2$. Then, $\adg(S)$ is connected and $\Diam(\adg(S))\leq 2$.
\end{proof}
\begin{proposition}\label{prop:adg-wright}
	Let $Z(S)\neq S$, $|Z(S)^{*}|\geq 2$, and every zero divisor of $S$ be a two-sided zero divisor. Then, $\ezdg(S)$ and $\adg(S)$ are connected. Moreover,
	$$
	\Diam(\adg(S))\leq\Diam(\zdg(S))\leq 3
	\quad\text{and}\quad
	\Diam(\ezdg(S))\leq\Diam(\zdg(S))\leq 3
	$$
\end{proposition}
\begin{proof}
	By the Proposition \ref{prop2}, the digraph $\zdg(S)$ is connected with diameter at most $3$. By Proposition \ref{propsub} and Theorem \ref{thm:zdg-in-adg}, $\zdg(S)$ is a spanning subdigraph of $\ezdg(S)$ and of $\adg(S)$, relatively.
\end{proof}
\begin{theorem}\label{thm:ge-conn}
	Let $|Z(S)^{*}|\geq 2$. Then the following statements are equivalent.
	\begin{enumerate}
		\item $\zdg(S)$ is connected.
		\item $\ezdg(S)$ is connected.
		\item Every zero divisor element of $S$ is a two-sided zero divisor.
	\end{enumerate}
\end{theorem}
\begin{proof}
	1 $\Rightarrow$ 2: If $\zdg(S)$ is connected, then $\ezdg(S)$ is connected because $\zdg(S)$ is a spanning subdigraph of $\ezdg(S)$.
	
	2 $\Rightarrow$ 3: Let $x\in Z(S)^{*}$. Since $|Z(S)^{*}|\geq 2$ and $\ezdg(S)$ is connected, the vertex $x$ must be the head of some arc $y \to x$ and the tail of some other arc $x \to w$ in $\ezdg(S)$. From the arc $y\to x$, there exist $m,n\in \mathbb{Z}^{+}$ such that $y^{m}x^{n}=0$, $y^{m}\neq 0$, and $x^{n}\neq 0$. Taking $n$ least, if $n=1$, then $y^{m}x=0$ with $y^{m}\neq 0$. If $n\geq 2$, then $(y^{m}x^{n-1})x=0$ with $y^{m}x^{n-1}\neq 0$ by the minimality of $n$. In both cases $x$ is a right zero divisor. By the symmetric argument, for the arc $x\to w$, $x$ is a left zero divisor. Hence every zero divisor of $S$ is two-sided.
	
	3 $\Rightarrow$ 1: This is the Proposition \ref{prop2}.
\end{proof}
\begin{lemma}\label{lem:degrees}
	Let $x\in Z(S)^{*}$. Then the out-degree of $x$ in $\zdg(S)$ is $|\Annr(x)\setminus\{0,x\}|$, and the in-degree of $x$ in $\zdg(S)$ is $|\Annl(x)\setminus\{0,x\}|$. If $Z(S)\neq S$, then the out-degrees of $x$ in $\ezdg(S)$ and in $\adg(S)$ are bounded below by $|\Annr(x)\setminus\{0,x\}|$, and the in-degrees of $x$ are bounded below by $|\Annl(x)\setminus\{0,x\}|$.
\end{lemma}
\begin{proof}
	If $y\in\Annr(x)\setminus\{0,x\}$, then $xy=0$ with $x\neq 0$. Thus, $y\in Z(S)^{*}$ and $x\to y$ is an arc of $\zdg(S)$. Moreover, the head of every arc leaving $x$ lies in $\Annr(x)\setminus\{0,x\}$. Then the out-degree of $x$ in $\zdg(S)$ is $|\Annr(x)\setminus\{0,x\}|$. The in-degree formula is similar. The lower bounds follow from Proposition \ref{propsub} and Theorem \ref{thm:zdg-in-adg}, respectively.
\end{proof}
\begin{corollary}\label{cor:girth}
	Let $Z(S)\neq S$. If $S$ is a reduced semigroup with $|Z(S)^{*}|\geq 2$, then $$\Girth(\zdg(S))=\Girth(\ezdg(S))=\Girth(\adg(S))=2$$
	
\end{corollary}
\begin{proof}
	Let $x\in Z(S)^{*}$. Then $xy=0$ or $yx=0$ for some $y\neq 0$, and $y\neq x$ since $x^{2}=0$ would make $x$ a nonzero nilpotent element. If $xy=0$, then $(yx)^{2}=y(xy)x=0$. Thus $yx=0$ since $S$ is reduced. The case $yx=0$ is similar. Hence, the cycle $x\to y\to x$ lies in $\zdg(S)$ and, by Proposition \ref{propsub} and Theorem \ref{thm:zdg-in-adg}, in $\ezdg(S)$ and $\adg(S)$, respectively.
\end{proof}
\begin{theorem}\label{thm:ge-cycle}
	The followings hold.
	\begin{enumerate}
		\item $\ezdg(S)$ contains a directed cycle if and only if $\zdg(S)$ contains a directed cycle.
		\item $\Girth(\zdg(S))\in\{2,3,4,\infty\}$ and $\Girth(\ezdg(S))\in\{2,3,4,\infty\}$, and $\Girth(\ezdg(S))\leq\Girth(\zdg(S))$.
	\end{enumerate}
\end{theorem}

\begin{proof}
	\begin{enumerate}
		\item Every directed cycle of $\zdg(S)$ is a directed cycle of $\ezdg(S)$. Conversely, let $$C= x_{1}\to x_{2}\to\cdots\to x_{k}\to x_{1}$$ be a directed cycle of $\ezdg(S)$ with $k\geq 2$ (read all indices modulo $k$). Since $C$ is a cycle, for each $i$, there exist $m_i, n_{i+1}$ positive integers such that $x_i^{m_i} x_{i+1}^{n_{i+1}}=0$ with $x_i^{m_i} \neq 0$ and $x_{i+1}^{n_{i+1}} \neq 0$. Take $i+1$, then $x_{i+1}^{m_{i+1}} x_{i+2}^{n_{i+2}}=0$ with $x_{i+1}^{m_{i+1}} \neq 0$ and $x_{i+2}^{n_{i+2}} \neq 0$. \\
		If $m_{i+1} > n_{i+1}$, since $x_i^{m_i} x_{i+1}^{n_{i+1}}=0$, $x_i^{m_i} x_{i+1}^{m_{i+1}}=0$  with $x_{i+1}^{m_{i+1}} \neq 0$ as well.
		Let $t_{i+1}=max\{n_{i+1},m_{i+1}\}$, for all $i$. Then, $x_{i}^{t_{i}}x_{i+1}^{t_{i+1}}=0$ and $x_{i}^{t_{i}}\neq 0$, for all $i$. Hence, $$
		x_1^{t_1} \to x_{2}^{t_{2}} \to \dots \to x_k^{t_k} \to x_1^{t_1}
		$$
		is a cycle of $\zdg(S)$.
		
		\item By Theorem \ref{theo4}, every arc of $\zdg(S)$ lying on a directed cycle of length greater than $4$ also lies on one of length at most $4$. Therefore, $\Girth(\zdg(S))\leq 4$ whenever $\zdg(S)$ contains a directed cycle, and $\Girth(\zdg(S))=\infty$ otherwise. If $\ezdg(S)$ contains a directed cycle, then so does $\zdg(S)$ by (i), and every directed cycle of $\zdg(S)$ lies in $\ezdg(S)$. Hence, $$\Girth(\ezdg(S))\leq\Girth(\zdg(S))\leq 4$$
	\end{enumerate}	
\end{proof}

\begin{definition}\label{def:knit}
	Let $S$ be a semigroup with zero, $V\subseteq S$, and $D=(V,E)$ be a digraph. A directed path $x_{1}\to x_{2}\to\cdots\to x_{n}$ in $D$ with $n\geq 2$ is a \textbf{directed left path} if $x_{1}\neq x_{n}$ and $x_{1}x_{i}=x_{n}x_{i}$, for all $i\in\{1,2,\ldots,n\}$. If $D$ contains a left path, the length of a shortest left path of $D$ is the \textbf{knit degree} $\kd(D)$.
\end{definition}

\begin{proposition}\label{prop:knit}
	Let $D$ be one of $\zdg(S)$, $\ezdg(S)$, and $\adg(S)$.
	\begin{enumerate}
		\item $\kd(D)=1$ if and only if $D$ contains an arc $x\to y$ with $x^{2}=yx$ and $xy=y^{2}$. 
		\item If $x,y\in Z(S)^{*}$ are distinct with $x^{2}=y^{2}=xy=yx=0$, then $\kd(\zdg(S))=\kd(\ezdg(S))=1$, and $\kd(\adg(S))=1$ whenever $Z(S)\neq S$.
		\item If $\zdg(S)$ contains a left path and $Z(S)\neq S$, then $\ezdg(S)$ and $\adg(S)$ contain left paths, and
		$$
		\kd(\adg(S))\leq \kd(\zdg(S)) \text{ and } \kd(\ezdg(S))\leq \kd(\zdg(S))
		$$
	\end{enumerate}
\end{proposition}

\begin{proof} Let $x,y \in Z(S)^*$ be distinct vertices.
	\begin{enumerate}
		\item Let $x\to y$ be a left path of length 1 in $D$. Then, $xx=y x$, $xy=y y$. Namely, $x^{2}=yx$ and $xy=y^{2}$. 
		\item By hypothesis, $xy=0$, namely $x \to y$ is an arc of $\zdg(S)$, and $y^2=0, x^2=0=yx$ imply that $x \to y$ is a left path. $x \to y$ is also left path of $\ezdg(S)$ because $\zdg(S)$ is a spanning subdigraph $\ezdg(S)$. Moreover, $x \to y$ is left path in $\adg(S)$ since $Z(S)\neq S$ by Theorem \ref{thm:zdg-in-adg}.
		\item  Every left path of $\zdg(S)$ is a directed path of $\ezdg(S)$ and, when $Z(S)\neq S$, of $\adg(S)$ with the same vertex sequence since the defining products are unchanged. Minimizing over left paths yields the inequalities.
	\end{enumerate}	
\end{proof}  

\begin{example}\label{ex:knit2}
	Let $S=\{0,a,b,c\}$ be the semigroup with $b^{2}=bc=cb=c^{2}=b$ and all other products equal to zero. Clearly, $\Annl(a)=\Annr(a)=S$, hence $Z(S)^{*}=\{a,b,c\}$. Then the graph $\zdg(S)$ is
	
	$$
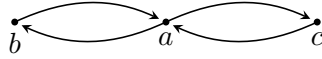
\begin{tikzpicture}[>=stealth, semithick]
		
		\coordinate (b) at (0,0);
		\coordinate (a) at (2,0);   
		\coordinate (c) at (4,0);
		
		\fill (b) circle (1.3pt) node[below] {$b$};
		\fill (a) circle (1.3pt) node[below] {$a$};
		\fill (c) circle (1.3pt) node[below] {$c$};

		\draw[shorten >=0.1cm,shorten <=0.01cm,->] (b) to[bend left=25] (a);
		\draw[shorten >=0.1cm,shorten <=0.01cm,->] (a) to[bend left=25] (b);
		\draw[shorten >=0.1cm,shorten <=0.01cm,->] (a) to[bend left=25] (c);
		\draw[shorten >=0.1cm,shorten <=0.01cm,->] (c) to[bend left=25] (a);
	\end{tikzpicture} $$
	Since $bs=cs$ for all $s\in S$, every directed path from $b$ to $c$ is a left path, and $b\to a\to c$ is a left path of length $2$. No left path of length $1$ exists: by Proposition \ref{prop:knit}(i). Hence, $\kd(\zdg(S))=2$.
	Here, $b^n c^m = b \neq 0$ for all $n,m \in \mathbb{Z}^+$. Thus, neither $b \to c$ nor $c \to b$ are arcs of $\ezdg(S)$. Then, $\ezdg(S)=\zdg(S)$ and $\kd(\ezdg(S))=2$ as well.
	Subsequently, a straightforward computation shows that $\Annl(xy)\setminus \Annl(x) \cup \Annl(y) = \emptyset$ and $\Annr(xy)\setminus \Annr(x) \cup \Annr(y) = \emptyset$, for all $x,y\in Z(S)^{*}$. Namely, $\adg(S)$ has no arcs.
\end{example}

%%%%%%%%%%%%%%%%%%%%%%%%%%%%%%%%%%%%%%%%%%%%%%%%%%%%%%%%%%%%%%%%%%%%%%%%%%%%%%%%
\section{Extended Zero-Divisor Digraphs of Rings}\label{sec:rings}

In this section, necessary and sufficient conditions for $\ezdg(R)=\zdg(R)$ to hold in unital rings are established and supported by examples and corollaries. Furthermore, for the noncommutative unital ring $M_n(F)$, a different necessary and sufficient condition ensuring the equality $\ezdg(M_n(F))=\zdg(M_n(F))$ is presented. In addition, the properties of one-sided identities are investigated.

\begin{theorem}\label{thm:ring-eq}
	Let $R$ be a unital ring. Then $\ezdg(R)=\zdg(R)$ if and only if the following two conditions hold.
	\begin{enumerate}
		\item $x^{2}=0$, for all $x\in\Nil(R)$.
		\item $\Annl(y^{2})=\Annl(y)$ and $\Annr(y^{2})=\Annr(y)$, for all $y\in Z(R)\setminus\Nil(R)$.
	\end{enumerate}
\end{theorem}

\begin{proof}
	($\Rightarrow$): Assume that $\ezdg(R)=\zdg(R)$.
	
	First, to prove that the item 1, suppose that there is an $x\in\Nil(R)$ with $n_{x}\geq 3$. If $n_{x}\geq 4$, then $\ezdg(R)\neq\zdg(R)$ by Corollary \ref{ggg} which is a contradiction. Thus, $n_{x}=3$. Let $y=x(1+x)=x+x^{2}$. If $y=0$, then $x=-x^{2}$, which implies $x^{2}=-x^{3}=0$, and if $y=x$, then $x^{2}=0$. Both contradict $n_x=3$. Since the elements $x$ and $1+x$ commute,
	$$
	y^{2}=x^{2}(1+x)^{2}=x^{2}+2x^{3}+x^{4}=x^{2}\neq 0
	\text{ and }
	y^{3}=x^{3}(1+x)^{3}=0
	$$
	so $y\in\Nil(R)\setminus\{0\}\subseteq Z(R)^{*}$. From $xy^{2}=x\,x^{2}=0$ with $x\neq 0$ and $y^{2}\neq 0$, hence $x\to y$ is an arc in $\ezdg(R)$. However,
	$$
	xy=x^{2}(1+x)=x^{2}+x^{3}=x^{2}\neq 0
	$$
	so $x\to y$ is not an arc in $\zdg(R)$, which contradicts $\ezdg(R)=\zdg(R)$. Consequently, $n_{x}=2$, for all $x\in\Nil(R)\setminus\{0\}$.
	
	To show that 2 holds, let $y\in Z(R)\setminus\Nil(R)$ and $0\neq s\in\Annl(y^{2})$. Since $sy^{2}=0$ with $y^{2}\neq 0$, the element $s \in Z(R)^{*}$. If $s=y$, then $y^{3}=0$, which contradicts $y\notin\Nil(R)$. Hence, $s\neq y$. From $s(y^{2})=0$ with $s\neq 0$ and $y^{2}\neq 0$, $s\to y$ is an arc in $\ezdg(R)=\zdg(R)$. Then, $sy=0$, namely $s\in\Annl(y)$. Therefore, $\Annl(y^{2})\subseteq\Annl(y)$, and $\Annl(y)\subseteq\Annl(y^{2})$ always holds. The equality $\Annr(y^{2})=\Annr(y)$ follows by a symmetric argument using the arc $y\to s$.
	
	($\Leftarrow$): 
	Since $n_{x}=2$, for all $x\in \Nil(S)$, without further computation, this is clear by Theorem \ref{thm:syc-eq}.
\end{proof}

\begin{corollary}\label{cor:ring-neq}
	Let $R$ be a unital ring. Then $\ezdg(R)\neq\zdg(R)$ if and only if there is an $x\in\Nil(R)$ with $n_{x}\geq 3$, or there is a $y\in Z(R)\setminus\Nil(R)$ with $\Annl(y^{2})\neq\Annl(y)$ or $\Annr(y^{2})\neq\Annr(y)$.
\end{corollary}

The following example demonstrates that the semigroup conditions cannot be extended directly to the ring setting. Although  $\zdg(S)=\ezdg(S)$ holds in the underlying semigroup, who contains a nilpotent $x$ with $n_x =3$, the ring $R$ constructed from the same semigroup $S$ fails to satisfy $\zdg(R)=\ezdg(R)$. 

\begin{example}\label{rem:semigroup-vs-ring}
	Let $S=\{0,x,x^{2}\}$ be the cyclic nilpotent semigroup with $x^{3}=0$. Here, $\zdg(S)$ is the complete digraph on two vertices. Since $\zdg(S)$ is a spanning subdigraph of $\ezdg(S)$, $\ezdg(S)$ is also the complete digraph on two vertices.
	
	Consider the ring $R = \{a +bx+cx^2 : a,b,c \in \mathbb{Z}, x \in S\}$. If $y=x+x^{2}$, since $x^2y = 0$ and $x^2 \neq 0$, then $y \in Z(R)^*$. This also implies that $x \to y$ is an arc of $\ezdg(R)$ but $xy=x(x+x^2)=x^2 \neq 0$ meaning $x \to y$ is not an arc of $\zdg(R)$. Thus, $\zdg(R)\neq\ezdg(R)$.
\end{example}

\begin{example}\label{ex:t2}
	Let $R=T_{2}(\mathbb{Z}_{2})$. Then
	
	$$Z(R)^*=\left\{\begin{bmatrix}
		\overline{1} & \overline{0} \\
		\overline{0} & \overline{0} 
	\end{bmatrix}, \begin{bmatrix}
		\overline{0} & \overline{1} \\
		\overline{0} & \overline{0} 
	\end{bmatrix}, \begin{bmatrix}
		\overline{0} & \overline{0} \\
		\overline{0} & \overline{1}
	\end{bmatrix}, \begin{bmatrix}
		\overline{1} & \overline{1} \\
		\overline{0} & \overline{0}
	\end{bmatrix}, \begin{bmatrix}
		\overline{0} & \overline{1}\\
		\overline{0} & \overline{1}
	\end{bmatrix}\right\}$$
	The only nonzero nilpotent element is $e_{12}$, and $n_{e_{12}}=2$. The remaining four vertices are idempotent, so condition (ii) of Theorem \ref{thm:ring-eq} holds. Hence, $\ezdg(R)=\zdg(R)$. A direct computation shows that $\zdg(R)= \ezdg(R)=\adg(R)$.
	
	\begin{figure}[h]
		\centering
		\begin{tikzpicture}[>=Stealth]
			
			\node (e12) at (0,2)
			{$\begin{bmatrix}
					\overline{0}&\overline{1}\\
					\overline{0}&\overline{0}
				\end{bmatrix}$};
			
			\node (e11) at (-3,0)
			{$\begin{bmatrix}
					\overline{1}&\overline{0}\\
					\overline{0}&\overline{0}
				\end{bmatrix}$};
			
			\node (e22) at (3,0)
			{$\begin{bmatrix}
					\overline{0}&\overline{0}\\
					\overline{0}&\overline{1}
				\end{bmatrix}$};
			
			\node (a) at (-1.3,-2.4)
			{$\begin{bmatrix}
					\overline{1}&\overline{1}\\
					\overline{0}&\overline{0}
				\end{bmatrix}$};
			
			\node (b) at (1.8,-2.4)
			{$\begin{bmatrix}
					\overline{0}&\overline{1}\\
					\overline{0}&\overline{1}
				\end{bmatrix}$};
			
			% Oklar
			
			\draw[->] (e12) -- (e11);
			\draw[->] (e12) -- (a);
			
			\draw[->,bend left=12] (e11) to (e22);
			\draw[->,bend left=12] (e22) to (e11);
			
			\draw[->] (b) -- (e11);
			\draw[->] (e22) -- (a);
			
			\draw[->,bend left=18] (a) to (b);
			\draw[->,bend left=18] (b) to (a);
			
			\draw[->] (b) to (e12);
			
			\draw [->] (e22) -- (e12);
			
		\end{tikzpicture}
		\caption{ $\zdg(T_{2}(\mathbb{Z}_{2}))=\ezdg(T_{2}(\mathbb{Z}_{2}))=\adg(T_{2}(\mathbb{Z}_{2}))$}
		\label{fig:t2}
	\end{figure}
\end{example}

\newpage

\begin{example}\label{ex:t3-counts}
	Let $R=T_{3}(\mathbb{Z}_{2})$. The element $x=\begin{bmatrix}
		\overline{0} & \overline{1} & \overline{0} \\
		\overline{0} & \overline{0} & \overline{1} \\
		\overline{0} & \overline{0} & \overline{0}
	\end{bmatrix}$ is nilpotent with $x^{2}=e_{13}\neq [\overline{0}]$ and $x^{3}=[\overline{0}]$; so $n_{x}=3$, and Corollary \ref{cor:ring-neq} yields $\ezdg(R)\neq\zdg(R)$. Furthermore, since $Z(R) \neq R$, by Theorem \ref{thm:zdg-in-adg}, $\zdg(R)$ is a spanning subdigraph of $\adg(R)$. 
	Let $y=\begin{bmatrix}
		\overline{0} & \overline{1} & \overline{0} \\
		\overline{0} & \overline{0} & \overline{1} \\
		\overline{0} & \overline{0} & \overline{0} 
	\end{bmatrix}$ and $z=\begin{bmatrix}
		\overline{1} & \overline{1} & \overline{0} \\
		\overline{0} & \overline{0} & \overline{1} \\
		\overline{0} & \overline{0} & \overline{0} 
	\end{bmatrix}$, then $yz= e_{13} \neq [\overline{0}]$. Thus, $y\to z$ is not an arc of $\zdg(R)$. However, since $\Annl(yz)=\left\{\begin{bmatrix}
		\overline{0} & \overline{b} & \overline{c} \\
		\overline{0} & \overline{d} & \overline{e} \\
		\overline{0} & \overline{0} & \overline{f} 
	\end{bmatrix} : \overline{b},\overline{c},\overline{d},\overline{e},\overline{f} \in \mathbb{Z}_2\right\}$ and $\Annl(y)=\left\{\begin{bmatrix}
		\overline{0} & \overline{0} & \overline{c} \\
		\overline{0} & \overline{0} & \overline{e} \\
		\overline{0} & \overline{0} & \overline{f} 
	\end{bmatrix} : \overline{c},\overline{e},\overline{f} \in \mathbb{Z}_2\right\}$, implying that $\Annl(yz)\setminus \Annl(y) \cup \Annl(z) \neq \emptyset$. Thus, $y \to z$ is an arc of $\adg(R)$ and $\adg(R) \neq \zdg(R)$.
	
	For the idempotent matrices $w=\begin{bmatrix}
		\overline{0} & \overline{0} & \overline{0} \\
		\overline{0} & \overline{1} & \overline{0} \\
		\overline{0} & \overline{0} & \overline{1} 
	\end{bmatrix}$ and $v=\begin{bmatrix}
		\overline{0} & \overline{0} & \overline{1} \\
		\overline{0} & \overline{0} & \overline{0} \\
		\overline{0} & \overline{0} & \overline{1} 
	\end{bmatrix}$, $wv= e_{33} \neq [\overline{0}]$. Hence, $w \to v$ is not an arc of $\ezdg(R)$. Let $r=\begin{bmatrix}
		\overline{1} & \overline{0} & \overline{0} \\
		\overline{0} & \overline{1} & \overline{0} \\
		\overline{0} & \overline{0} & \overline{0} 
	\end{bmatrix}$. Since $r(wv)=[\overline{0}]$, $rw=e_{22}$, and $rv=e_{13}$, then $r \in \Annl(wv) \setminus (\Annl(w) \cup \Annr(v))$. Namely, $w \to v$ is an arc of $\adg(R)$, so $\adg(R) \neq \ezdg(R)$.
	Hence, three digraphs are pairwise distinct.
\end{example}

Here $n \in \mathbb{Z}^+$ with $n \geq 2$, and $M_{n}(F)$ denotes the unital ring of $n\times n$ matrices over a field $F$.

\begin{theorem}\label{thm:matrix}
	$\ezdg(M_{n}(F))=\zdg(M_{n}(F))$ if and only if $n=2$.
\end{theorem}

\begin{proof}
	$(\Leftarrow):$ Let $n=2$ and $A\in Z(M_{2}(F))^{*}$. Thus, $\det A=0$, and the Cayley--Hamilton theorem yields
	$$
	A^{2}=\operatorname{tr}(A)A
	$$
	If $A\in\Nil(M_{2}(F))$, then both eigenvalues of $A$ are zero, so $\operatorname{tr}(A)=0$ and $A^{2}=0$. Hence, condition 1 of Theorem \ref{thm:ring-eq} holds. If $A\in Z(M_{2}(F))\setminus\Nil(M_{2}(F))$, then $c=\operatorname{tr}(A)\neq 0$, since $c=0$ would yield $A^{2}=0$, thus $A^{2}=cA$ and
	$$
	SA^{2}=0 \Longleftrightarrow c(SA)=0 \Longleftrightarrow SA=0
	\text{ and }
	A^{2}S=0 \Longleftrightarrow AS=0
	$$
	for all $S\in M_{2}(F)$. As a result, $\Annl(A^{2})=\Annl(A)$ and $\Annr(A^{2})=\Annr(A)$, and condition 2 of Theorem \ref{thm:ring-eq} holds. Consequently, $\ezdg(M_{2}(F))=\zdg(M_{2}(F))$.
	
	$(\Rightarrow):$ For $n\geq 3$, the matrix $N=e_{12}+e_{23}+\cdots+e_{n-1,n}$ satisfies $N^{n-1}=e_{1n}\neq 0$ and $N^{n}=0$. Therefore $N$ is nilpotent of index $n\geq 3$, and Corollary \ref{cor:ring-neq} yields $\ezdg(M_{n}(F))\neq\zdg(M_{n}(F))$.
\end{proof}

\begin{corollary}\label{cor:matrix-diam}
	\begin{enumerate}
		\item $\ezdg(M_{n}(F))$ is connected, and $$\Diam(\ezdg(M_{n}(F)))=\Girth(\ezdg(M_{n}(F)))=2$$
		\item $\adg(M_{n}(F))$ is connected, and
		$$\Diam(\adg(M_{n}(F)))=\Girth(\adg(M_{n}(F)))=2$$
	\end{enumerate}	
\end{corollary}

\begin{proof}
	By Lemma \ref{lem2}, $\zdg(M_{n}(F))$ is connected with diameter $2$. Then $\ezdg(M_{n}(F))$ and $\adg(M_{n}(F))$ are connected with diameter at most $2$, since $\zdg(M_{n}(F))$ is a spanning subdigraph of $\ezdg(M_{n}(F))$ and $\adg(M_{n}(F))$ by Proposition \ref{propsub}, and Corollary \ref{cor:monoid-adg}, respectively. \\
	Let $A=e_{11}$, $B=\begin{bmatrix}
		1 & 1 & 0 & \dots &0 \\
		0 & 0 & 0 & \dots &0 \\
		0 & 0 & 0 & \dots &0\\
		\vdots & \vdots & \vdots & \dots &\vdots\\
		0 & 0 & 0 & \dots &0
	\end{bmatrix}\in M_{n}(F)$. Then $A$ and $B$ are singular, nonzero, and idempotent. Thus, $A^{i}B^{j}=AB=B\neq 0$, for all $i,j\in\ZP$, namely $A\to B$ is not an arc in $\ezdg(M_{n}(F))$. Hence, $\Diam(\ezdg(M_{n}(F)))=2$. Moreover, $e_{11}e_{22}=0=e_{22}e_{11}$ yields the directed cycle $e_{11}\to e_{22}\to e_{11}$. Hence,
	$$
	\Girth(\zdg(M_{n}(F)))=\Girth(\ezdg(M_{n}(F)))=2.
	$$
	Moreover, $\Annl(AB)\setminus \Annl(A)\cup \Annl(B)=\emptyset$ and $\Annr(AB)\setminus \Annr(A)\cup \Annr(B)=\emptyset$. Hence, $A\to B$ is not an arc in $\adg(M_{n}(F))$, and the diameter equals $2$. Finally, the cycle $e_{11}\to e_{22}\to e_{11}$ of $\zdg(M_{n}(F))$ lies in $\adg(M_{n}(F))$. Thus, the girth equals $2$.
\end{proof}

\begin{example}\label{rem:matrix-adg}
	By Theorem \ref{thm:matrix}, since $n=2$, $\ezdg(M_{2}(\mathbb{Z}_{2}))=\zdg(M_{2}(\mathbb{Z}_{2}))$.
	
	Let $x=\begin{bmatrix}
		\overline{0}&\overline{0}\\
		\overline{1}&\overline{1}
	\end{bmatrix}, y=\begin{bmatrix}
		\overline{1}&\overline{1}\\
		\overline{0}&\overline{0}
	\end{bmatrix}$. Since $xy=x\neq [\overline{0}]$, $x \to y$ is not an arc of $\zdg(M_{2}(\mathbb{Z}_{2}))$. Therefore, the $\Diam(\zdg(M_{2}(\mathbb{Z}_{2})))=2=\Diam(\ezdg(M_{2}(\mathbb{Z}_{2})))$. Moreover, since $Z(M_{2}(\mathbb{Z}_{2}))\neq M_{2}(\mathbb{Z}_{2})$, by Theorem \ref{thm:zdg-in-adg}, $\zdg(M_{2}(\mathbb{Z}_{2}))$ is a spanning subdigraph of $\adg(M_{2}(\mathbb{Z}_{2}))$. Here, for vertices $e_{12}$ and $e_{22}$, since $e_{12}e_{22}=e_{12}$, $e_{12}\to e_{22}$ is not an arc of $\adg(M_{2}(\mathbb{Z}_{2}))$. Thus, $\Diam(\adg(M_{2}(\mathbb{Z}_{2})))=2$. Furthermore, since $e_{11}e_{22}=[\overline{0}]=e_{22}e_{11}$, then $$\Girth(\adg(M_{2}(\mathbb{Z}_{2})))=\Girth(\zdg(M_{2}(\mathbb{Z}_{2})))=\Girth(\ezdg(M_{2}(\mathbb{Z}_{2})))=2$$
\end{example}

\begin{corollary}\label{cor:ring-2abs}
	Let $R$ be a noncommutative unital ring such that $abc=0$ implies $ab=0$, $bc=0$, or $ac=0$, for all $a,b,c\in R$. Then
	$$
	\zdg(R)=\ezdg(R)=\adg(R)
	$$
\end{corollary}

\begin{proof}
	Since $R$ is unital, $Z(R)\neq R$, and the proof is clear by Corollary \ref{cor:2abs-equal}.
\end{proof}

\begin{theorem}\label{thm:ring-conn}
	Let $R$ be an artinian noncommutative ring. Then the following statements are equivalent.
	\begin{enumerate}
		\item $\ezdg(R)$ is connected.
		\item $\zdg(R)$ is connected.
		\item Every one-sided identity element of $R$ is a two-sided identity of $R$.
	\end{enumerate}
\end{theorem}

\begin{proof}
	1 $\Rightarrow$ 3: Suppose that a one-sided identity element is not a two-sided identity. Assume that $R$ has a left identity $e$. Then there is an $a\in R$ with $x=ae-a\neq 0$, and $xb=a(eb)-ab=0$, for all $b\in R$. In particular, $xe=0$, that is $e,x\in Z(R)^{*}$, and $x\neq e$ since $x=e$ would yield $e=e^{2}=xe=0$. Every vertex $z\neq e$ satisfies $ez=z\neq 0$. Then, $e$ is a sink of $\zdg(R)$, and Theorem \ref{thm:sinks}(i) yields that $e$ is a sink of $\ezdg(R)$. Since $e$ is a sink of $\ezdg(R)$, $\ezdg(R)$ is not connected. This is a contradiction.	
	
	3 $\Rightarrow$ 2: This is the Theorem \ref{theo24}.
	
	2 $\Rightarrow$ 1: This is the trivial direction of Theorem \ref{thm:ge-conn}: a digraph containing a connected spanning subdigraph is connected. 
\end{proof}
\begin{proposition}\label{prop:identity}
	Let $S$ be a semigroup with zero and $e\in Z(S)^{*}$ be a left (right) identity element of $S$. Then $e$ has out (in)-degree zero in each of $\zdg(S)$, $\ezdg(S)$, and $\adg(S)$. 
\end{proposition}

\begin{proof}
	Let $y\in Z(S)^{*}\setminus\{e\}$. Then $ey=y\neq 0$, implying $e\to y$ is not an arc in $\zdg(S)$. The products $e^{a}y^{b}=y^{b}$ are nonzero whenever $y^{b}\neq 0$, for all $b\in \mathbb{Z}^{+}$. Thus, $e\to y$ is not an arc in $\ezdg(S)$. For $\adg(S)$, since $ey=y$, $\Annl(ey)\setminus \Annl(e)\cup \Annl(y)=\emptyset$ and $\Annr(ey)\setminus \Annr(e)\cup \Annr(y)=\emptyset$. 
\end{proof}

\begin{corollary}\label{cor:identity-disc}
	Let $R$ be a noncommutative ring with a one-sided identity element. Then none of $\zdg(R)$, $\ezdg(R)$, and $\adg(R)$ is connected. In particular, if $\adg(R)$ is connected, then every one-sided identity of $R$ is a two-sided identity of $R$.
\end{corollary}

\begin{proof}
	Assume that $R$ has a left identity element $e$. As in the proof of Theorem \ref{thm:ring-conn}, there is an $a\in R$ with $x=ae-a\neq 0$ and $xb=0$, for all $b\in R$. This implies that $e,x\in Z(R)^{*}$ and $x\neq e$. All three digraphs have the vertex set $Z(R)^{*}$, and Proposition \ref{prop:identity} yields that $e$ has out-degree zero in each of them. Therefore, no directed walk joins $e$ to the further vertex $x$, and none of the three digraphs is connected.
\end{proof}
%%%%%%%%%%%%%%%%%%%%%%%%%%%%%%%%%%%%%%%%%%%%%%%%%%%%%%%%%%%%%%%%%%%%%%%%%%%%%%%%
\section{Conclusion}\label{sec:conc}

This paper investigated the generalizations of the zero-divisor graph developed in \cite{am20} for digraphs of semigroups and rings, in light of the extended zero-divisor digraph introduced in \cite{syc}. The annihilator digraph $\adg(S)$ was introduced for a semigroup $S$ with zero, and any two distinct nonzero nilpotent elements were demonstrated to be joined by directed paths of length at most $2$ in $\zdg(S)$, which yields $\Diam(\zdg(S))\leq 2$ when $Z(S)=\Nil(S)$ and sharpens the bound of \cite{w07} in this case. The end vertices of $\ezdg(S)$ were identified with those of $\zdg(S)$. Combined with the results of Wu \cite{w05}, this yields a characterization of the connectedness of $\ezdg(R)$ for Artinian noncommutative rings in terms of the existence of one-sided identity elements. Moreover, one-sided identity elements were shown to be end vertices of all three digraphs, extending the disconnectedness result to $\adg(R)$. The equalities $\ezdg(M_{n}(F))=\zdg(M_{n}(F))$ holds if and only if $n=2$, over any field $F$. Furthermore, $\adg(M_{n}(F))$ and $\ezdg(M_{n}(F))$ shown to be connected with diameter and girth $2$. Subsequently, in the ring $T_{3}(\mathbb{Z}_{2})$, it also shown that $\zdg(T_{3}(\mathbb{Z}_{2}))$ is a spanning subdigraph of $\adg(T_{3}(\mathbb{Z}_{2}))$, but $\ezdg(T_{3}(\mathbb{Z}_{2}))$ is not a subdigraph of $\adg(T_{3}(\mathbb{Z}_{2}))$. The connectedness, diameter, girth, and vertex degrees of $\adg(S)$ were bounded when every zero divisor element is nilpotent or two-sided zero divisor. It was also demonstrated that the graph $\zdg(S)$ is connected if and only if $\ezdg(S)$ is connected under the condition $|(Z(S))^{*}|\geq 2$. Moreover, $\zdg(S)$ contains a directed cycle if and only if $\zdg(S)$ contains a directed cycle. The necessary conditions for the equality of the digraphs $\zdg(S)$, $\ezdg(S)$ and $\adg(S)$ have also been established. For unital rings, the equality $\ezdg(R)=\zdg(R)$ was characterized using nilpotent elements. 

Future studies can characterize the semigroups and rings satisfying $\adg(S)=\zdg(S)$ or $\ezdg(S)$ is a subdigraph of $\adg(S)$, and determine whether two distinct nonzero nilpotent elements of a noncommutative ring are necessarily adjacent in $\adg(R)$. In particular, it remains open whether $\adg(M_{n}(F))=\zdg(M_{n}(F))$ holds, as well as how the arcs of $\adg(M_{n}(F))$ can be characterized in terms of matrix ranks. Whether the sinks and the sources of $\zdg(S)$ remain end vertices of $\adg(S)$ when $Z(S)\neq S$, in analogy with Theorem \ref{thm:sinks}, is also open. The same is true of the item (i) left open by Theorem \ref{thm:ge-cycle}.

%\backmatter
\subsubsection*{Author Contributions}

All authors contributed equally.

\subsubsection*{Financial Disclosure}

None reported.

\subsubsection*{Conflicts of Interest}

The authors declare no conflicts of interest.

\bibliographystyle{plain}
\bibliography{refs}

\end{document}